\documentclass[11pt, reqno]{amsart}

\usepackage[utf8]{inputenc}
\usepackage[T2A]{fontenc}
\usepackage[ukrainian,english]{babel}

\usepackage[english]{d-omega3}
\usepackage{amssymb}
\usepackage{graphicx}
\usepackage[all]{xy}

\begin{document}

%%%%%%%%%%%%%%%% Insert metadata of the paper

% \selectlanguage{ukrainian}
\selectlanguage{english}

% Set the title of the paper:
% If yout title is long and will not fit to hte top margin of the paper
% then specify a short title in square brackets
\title[Sets of distinct representations in redundant numeral systems]{Sets of distinct representations of numbers in numeral systems with a natural base and a redundant alphabet}

% Your title may contain line breaks like '\\'
% In that case define the following command including there your title without such symbols '\\'
% \contentsTitle{}

%%% information on the first author:
%%  \author[Name which will be printed at the end of the paper]
%%         {Name printed in the paper title}
\author[M. Pratsiovytyi]{Mykola Pratsiovytyi}
%\department{Department of Algebra}
\organization{Institute of Mathematics of NAS of Ukraine, Dragomanov Ukrainian State University}
%\address{Ukraine}
\email{prats4444@gmail.com}
\orcid{0000-0001-6130-9413}

%%%%%%%%%%%% information on the second author:
\author[O. Vynnyshyn]{Oleh Vynnyshyn}
% \department{...}
\organization{Institute of Mathematics of NAS of Ukraine}
% \address{...}
\email{oleh.vynnyshyn@imath.kiev.ua}
\orcid{0009-0006-7113-1501}

%%%%%%%%%%%% general information about the paper
% list of authors in the header of the paper
\shortAuthorsList{M.~Pratsiovytyi, O.~Vynnyshyn}

%%%%%%%%%%%% similarly add information about all other authors

%%%%%%%%%%%% abstracts
\abstract{english}{
In this work, we study a numeral system with a natural base $s \geq 2$ and a redundant alphabet $A_r=\{0,1, \dots, r\}$, where $s \leq r \leq 2s-2$. We investigate the topological, metric, and fractal properties of the set of numbers in the interval $\left[0,\frac{r}{s-1}\right]$ that admit a unique representation $x=\sum\limits_{n=1}^{\infty}\frac{\alpha_n} {s^n}\equiv\Delta^{r_s}_{\alpha_1\alpha_2...\alpha_n...}$, $\alpha_n\in A_r$. The criterion for the uniqueness of the number representation is established. It is proved that the Hausdorff--Besicovitch dimension of the set of numbers with a unique representation is equal to $\frac{\ln(2s-r-1)}{\ln s}$. An analysis of the quantity of representations of numbers having purely periodic representations
with a simple period (a single-digit period) is carried out. It is proved that the set of numbers that admit a continuum of distinct representations has full Lebesgue measure. Conditions for a number to belong to this set are given in terms of one of its representations.}

%\abstract{ukrainian}{У роботі для системи числення з натуральною основою $s \geq 2$ і надлишковим алфавітом $A_r=\{0,1, \dots, r\}$, $s\leq r\leq 2s-2$, вивчаються тополого-метричні і фрактальні властивості множин чисел відрізка $[0;\frac{r}{s-1}]$, які мають єдине зображення: $x=\sum\limits_{n=1}^{\infty}\frac{\alpha_n}{s^n}\equiv\Delta^{r_s}_{\alpha_1\alpha_2...\alpha_n...}$, $\alpha_n\in A_r$. Обґрунтовано критерій єдиності зображення числа, доведено, що розмірність Гаусдорфа-Безиковича множини чисел з єдиним зображенням рівна $\frac{\ln(2s-r-1)}{\ln s}$. Наведено результати аналізу кількості зображень у чисел, які мають чисто періодичне зображення з простим періодом (періодом з однієї цифри). Доведено, що множина чисел, які мають континуальну множину різних зображень, є множиною повної міри Лебега. Наведено умови належності чисел цій множині за одним з їх зображень.}

% \dedication{...}
\thanks{This work was supported by a grant from the Simons Foundation (SFI-PD-Ukraine-00014586, P.M., V.O.).}
\keywords{A numeral system with a natural base and a redundant alphabet; numbers that have a single representation; the set of numbers that have a finite quantity of representations; Cantor set type; the Hausdorff-Besicovitch dimension. %Cистема числення з натуральною основою та надлишковим алфавітом; числа, які мають єдине зображення; множина чисел, що мають континуальну кількість зображень; множина канторівського типу, розмірність Гаусдорфа-Безиковича
}
\msc{28A80, 11A67, 28E15}
%
% \doi{}
% \received{}
% \accepted{}

\maketitle

%%% Text of the paper

\section{Introduction}
Let $s,r$ be fixed natural numbers, $2 \leq s \leq r \leq 2s-2$; let  
$A_r = \{0,1,\dots,r\}$ be an alphabet (a set of digits), and  
$L_r = A_r \times A_r \times \cdots$ be the set of sequences with elements from the alphabet.  
It is known~\cite{3_redundant_alph_Honcharenko,1_dist_rv} that for any number 
$x \in \left[0;\frac{r}{s-1}\right]$ there exists a sequence $(a_k) \in L_r$ such that
\begin{equation}\label{ser1}
  x = \sum_{k=1}^{\infty} a_k s^{-k} =
  \Delta^{r_s}_{a_1 a_2 \dots a_k \dots}.
\end{equation}
The symbolic (abbreviated) notation
$x = \Delta^{r_s}_{a_1 a_2 \dots a_n \dots}$ of the number expansion into the series~\eqref{ser1}
is called the \emph{$r_s$-representation} of the number $x$, or the representation of $x$
in the numeral system with base $s$ and alphabet $A_r$.
Hereafter, parentheses in the representation of a number denote a period.
Obviously,
\begin{equation}\label{rozkladdelta}
\Delta^{r_s}_{a_1a_2 \ldots }=
\Delta^{r_s}_{a_1a_2 \ldots a_m(0)}+
\Delta^{r_s}_{0\ldots0a_{m+1}a_{m+2}\ldots}
=
\Delta^{r_s}_{a_1a_2 \ldots a_m(0)}+\frac{1}{s^m}
\Delta^{r_s}_{a_{m+1}a_{m+2}\ldots}, 
\end{equation}
\[\Delta^{r_s}_{\alpha_1\alpha_2 \ldots \alpha_n...}+
\Delta^{r_s}_{\beta_1\beta_2 \ldots \beta_n \ldots}=
\Delta^{r_s}_{[\alpha_1+\beta_1][\alpha_2+\beta_2] 
\ldots [\alpha_n+\beta_n] \ldots}, 
[\alpha_n+\beta_n]\in A_r.\]
If
\[\frac{a}{s} + \frac{b}{s^2} = \frac{c}{s} + \frac{d}{s^2},
\quad \text{which is equivalent to} \quad s(a-c)=d-b,\]
then the pairs $(a,b)$ and $(c,d)$ in the representation of a number are  mutually replaceable (alternative).

The redundancy of the alphabet $A_r$ guarantees existence of numbers that have an infinite set and even a continuum of different $r_s$-representations~\cite{14_Komornik}.
Such systems of encoding real numbers~\cite{10_encod_systems} have been studied in various works,
in particular in~\cite{3_redundant_alph_Honcharenko,15_Mykytiuk}.
However,  we did not find a complete answer to the question of the count of different representations
of a number from the interval $\left[0;\frac{r}{s-1}\right]$.
Therefore, we continue the previously initiated research, specifying certain results
and providing more detailed arguments.

Note that numeral systems with a natural base and a redundant alphabet have been used
in the distribution theory of random variables, in particular for infinite Bernoulli
convolutions~\cite{1_dist_rv,11_convolutions,13_Makarchuk},
in fractal theory~\cite{4_Pratsiovytyi},
in the geometry of numerical series~\cite{0_vmps_kantorval,12_Karvatskyi2024},
in the theory of locally complex functions with fractal properties~\cite{9_webtype_curve},
and so on. They arise naturally at the intersection of these theories.

\section{Numbers with purely periodic representations and a simple period}

\begin{lemma}
Every number $x_0 \in \left[0;\frac{r}{s-1}\right]$ that has a purely periodic $r_s$-representation
with period $(c)$, where
$c \in \{0,r\} \cup \{r-s+2, \ldots, s-2\}$, has a unique representation.
\end{lemma}

\begin{proof}
If $c \in \{0, r\}$, then the statement is obvious.
Consider the number $x_0 = \Delta^{r_s}_{(r-s+1+m)}$, where
$m \in \{1,2,\ldots,2s-r-3\}$. Assume that there exists another $r_s$-representation
$\Delta^{r_s}_{\alpha_1 \alpha_2 \ldots \alpha_n \ldots}$ of the number $x_0$.
Since the representations are different, there is the smallest $k$ such that
$\alpha_k \neq c$.
We proceed by considering the following cases:

\medskip

1. If $\alpha_k > c$, then, taking into account~\eqref{rozkladdelta}, we have
\[\Delta^{r_s}_{\alpha_{k} \alpha_{k+1} \alpha_{k+2} \ldots} \geq 
  \Delta^{r_s}_{\alpha_{k}(0)} \geq
  \Delta^{r_s}_{[c+1](0)} =
  \Delta^{r_s}_{c(s-1)} >
  \Delta^{r_s}_{c(c)},\]
since $c < s-1$.
This yields a contradiction to the fact that
$\Delta^{r_s}_{\alpha_1 \alpha_2 \ldots \alpha_n \ldots}$ and
$\Delta^{r_s}_{(c)}$
are different representations of the same number.

\medskip

2. If $\alpha_k < c$, then
\[\Delta^{r_s}_{\alpha_{k} \alpha_{k+1} \alpha_{k+2} \ldots} \leq 
\Delta^{r_s}_{\alpha_{k}(r)} \leq
\Delta^{r_s}_{[c-1](r)} =
\Delta^{r_s}_{c(r-(s-1))} <
\Delta^{r_s}_{c(c)},\]
since $c > r-(s-1)$.
This again leads to a contradiction that
$\Delta^{r_s}_{\alpha_1 \alpha_2 \ldots \alpha_n \ldots}$ and
$\Delta^{r_s}_{(c)}$
are representations of the same number.
\end{proof}
\begin{lemma}\label{periodic_continuum}
Every number $x_0 \in \left[0;\frac{r}{s-1}\right]$ that has a purely periodic $r_s$-representation
with period $(c)$, where
$c \in \{1,2,\ldots,r-s,s,s+1,\ldots,r-1\}$,
has a continuum of distinct representations.
\end{lemma}

\begin{proof}
If $c \in \{1,2,\ldots,r-s\}$, then for each pair of adjacent digits $cc$
there exists a mutually replaceable pair $(c-1, c+s)$.
If $c \in \{s,s+1,\ldots,r-1\}$, then for the pair $cc$
the alternative one is $(c+1, c-s)$.
Since there are infinitely many positions of adjacent digit pairs
in the $r_s$-representation of the number $x_0$ with alternatives,
we conclude that there exists a continuum of distinct representations of $x_0$.
\end{proof}
\begin{lemma}\label{periodic_infinite}
Every number $x_0 \in \left[0;\frac{r}{s-1}\right]$ that has a purely periodic $r_s$-representation
with period $(c)$, where
$c \in \{r-s+1, s-1\}$,
has an infinite set of distinct $r_s$-representations.
\end{lemma}

\begin{proof}
If $c = s-1$, then the equalities
    \[\Delta^{r_s}_{(s-1)} =
    \Delta^{r_s}_{s(0)} = 
    \Delta^{r_s}_{[s-1]s(0)}=
    \Delta^{r_s}_{[s-1][s-1]\ldots[s-1]s(0)}\]
hold, and the given number has an infinite set of representations.
If $c = r-s+1$, then, similarly,
    \[\Delta^{r_s}_{(r-s+1)} =
    \Delta^{r_s}_{[r-s](r)} =
    \Delta^{r_s}_{[r-s+1][r-s](r)}  =
    \Delta^{r_s}_{[r-s+1]\ldots[r-s+1][r-s](r)}.\]
Thus, the number $x_0$ has an infinite set of representations.
\end{proof}

\section{Numbers with a unique representation}

The following statement is obvious.

\begin{lemma}\label{uniq_inversed}
The number
$x = \Delta^{r_s}_{\alpha_1 \alpha_2 \ldots \alpha_n \ldots}$
has a unique $r_s$-representation if and only if the number
$x' = \Delta^{r_s}_{[r-\alpha_1][r-\alpha_2]\ldots[r-\alpha_n]\ldots}$
has a unique representation.
\end{lemma}

Indeed, if the number $x$ has another representation
$\Delta^{r_s}_{\beta_1 \ldots \beta_n \ldots}$, where
$\beta_k \neq \alpha_k$ for some $k$, then
$\Delta^{r_s}_{[r-\beta_1][r-\beta_2]\ldots[r-\beta_n]\ldots}$
is another representation of the number $x'$.
Conversely, the same holds in the opposite direction.
This contradicts the uniqueness of the $r_s$-representation of the number.

\begin{theorem}[Criterion of the uniqueness of a number representation]
\label{criterion_of_uniq}
For a number
$x = \Delta^{r_s}_{\alpha_1 \alpha_2 \ldots \alpha_n \ldots}$
in the interval $\left(0,\frac{r}{s-1}\right)$
to have a unique $r_s$-representation, it is necessary and sufficient that,
for every $k \in \mathbb{N}$, the following two conditions hold:

\begin{enumerate}
\item
$\Delta^{r_s}_{\alpha_{k+1} \alpha_{k+2} \ldots \alpha_{k+n} \ldots} < 1$
or $\alpha_k = r$;

\item
$\Delta^{r_s}_{[r-\alpha_{k+1}] [r-\alpha_{k+2}] \ldots [r-\alpha_{k+n}] \ldots} < 1$
or $\alpha_k = 0$.
\end{enumerate}
\end{theorem}

\begin{proof}
\textbf{Necessity.}
Let the number
$x = \Delta^{r_s}_{\alpha_1 \alpha_2 \ldots \alpha_n \ldots}$
have a unique $r_s$-representation.
We prove that conditions 1) and 2) are met.

Since $0 \neq x \neq \frac{r}{s-1}$, there exists $k$ such that
$\alpha_k \neq r$.
Consider the number
$x_k = \Delta^{r_s}_{\alpha_{k+1} \ldots \alpha_{k+n} \ldots}$.
Assume that $x_k \geq 1$.

If $x_k = 1$, then
$x_k = \Delta^{r_s}_{(s-1)} =
\Delta^{r_s}_{[s-1]s(0)}$,
and hence
$x = \Delta^{r_s}_{\alpha_1 \ldots \alpha_k (s-1)} =
\Delta^{r_s}_{\alpha_1 \ldots \alpha_k [s-1]s(0)}$,
which contradicts the uniqueness of the $r_s$-representation of $x$.

Assume that $x_k > 1$.
Then among the digits of the representation of $x_k$
there exists $\alpha_{k+j} > s-1$; we take the smallest such $j$.
If $j = 1$, then
\[ x = \Delta^{r_s}_{\alpha_1 \ldots \alpha_{k-1}
[\alpha_k+1][\alpha_{k+1}-s]\alpha_{k+2}\alpha_{k+3}\ldots}. \]

If $\alpha_{k+1} \leq s-1$, then
\[ x = \Delta^{r_s}_{\alpha_1 \ldots \alpha_k \alpha_{k+1} \ldots
[\alpha_{k+j-1}+1] [\alpha_{k+j}-s] \alpha_{k+j+1}\ldots}. \]
In both cases we obtain another $r_s$-representation of the number $x$,
which contradicts its uniqueness.
These contradictions prove condition~1).

Condition~2) is proved analogously. Moreover, it follows directly
from condition~1) and the previous lemma.

\medskip

\textbf{Sufficiency.}
Assume that for the number
$x = \Delta^{r_s}_{\alpha_1 \alpha_2 \ldots \alpha_n \ldots}$
both conditions~1) and~2) hold for every $k \in \mathbb{N}$,
and that there exists another representation
$x = \Delta^{r_s}_{\beta_1 \beta_2 \ldots \beta_n \ldots}$.
Then there exists $k \in \mathbb{N}$ such that
$\alpha_k \neq \beta_k$, while $\alpha_j = \beta_j$ for all $j < k$.

Suppose that $\alpha_k > \beta_k$.
Then $\alpha_k \neq 0$, and by condition~2) we have
\[\Delta^{r_s}_{(r)} - \Delta^{r_s}_{\alpha_{k+1} \alpha_{k+2} \ldots} 
=\Delta^{r_s}_{[r - \alpha_{k+1}] [r-\alpha_{k+2}] \ldots [r-\alpha_{k+n}] \ldots} < 1.\]
Then
$\Delta^{r_s}_{(r)}-1<\Delta^{r_s}_{\alpha_{k+1} \alpha_{k+2} \ldots \alpha_{k+n} \ldots}$, but 
\[\Delta^{r_s}_{(r)}-1=
\Delta^{r_s}_{(r)}-\Delta^{r_s}_{(s-1)}=\Delta^{r_s}_{(r-s+1)}<
\Delta^{r_s}_{\alpha_{k+1} \alpha_{k+2} \ldots \alpha_{k+n} \ldots}.\]

Hence
\[\Delta^{r_s}_{\alpha_k(r-s+1)}=
\frac{\alpha_k}{s} + 
\frac{1}{s}\Delta^{r_s}_{(r-s+1)} <
\frac{\alpha_k}{s} + 
\frac{1}{s}\Delta^{r_s}_{\alpha_{k+1} \ldots \alpha_{k+n} \ldots}=
\Delta^{r_s}_{\alpha_{k} \alpha_{k+1} \ldots \alpha_{k+n} \ldots}.\]

However, 
\[\Delta^{r_s}_{\alpha_k(r-s+1)}=
\Delta^{r_s}_{[\alpha_{k}-1](r)},\]
and therefore 
\[\Delta^{r_s}_{[\alpha_{k}-1](r)}<\Delta^{r_s}_{\alpha_{k} \alpha_{k+1}\ldots \alpha_{k+n} \ldots}.\]

Consequently,
\[ \Delta^{r_s}_{\beta_k \beta_{k+1} \beta_{k+2} \ldots} \leq 
\Delta^{r_s}_{\beta_k(r)} \leq 
\Delta^{r_s}_{[\alpha_k -1](r)} < 
\Delta^{r_s}_{\alpha_{k}\alpha_{k+1} \alpha_{k+2} \ldots},\]
and hence
\[\Delta^{r_s}_{\beta_k \beta_{k+1} \beta_{k+2} \ldots} < 
\Delta^{r_s}_{\alpha_{k}\alpha_{k+1} \alpha_{k+2} \ldots}.\]

Thus, in the case $\alpha_k > \beta_k$, we obtain a contradiction of the assumption that there exists another representation of the number $x$.

Now suppose that $\alpha_k < \beta_k$.
By condition~1) we have
\[\Delta^{r_s}_{\alpha_k \alpha_{k+1} \alpha_{k+2} \ldots} <
\Delta^{r_s}_{[\alpha_k +1](0)}\]
and $\alpha_k \neq r$. 

Moreover, 
\[\Delta^{r_s}_{[\alpha_k +1](0)} \leq 
\Delta^{r_s}_{\beta_k \beta_{k+1} \beta_{k+2} \ldots},\]
and hence 
\[\Delta^{r_s}_{\alpha_k \alpha_{k+1} \alpha_{k+2} \ldots} < 
\Delta^{r_s}_{\beta_k \beta_{k+1} \beta_{k+2} \ldots}.\]

Thus, in this case as well, we obtain a contradiction of the assumption that there exists another representation of the number $x$. The sufficiency and hence the entire theorem are proved.
\end{proof} 

\begin{remark}
If the $r_s$-representation of a number $x$ has period $(s-1)$ or $(r-s+1)$, then $x$ has infinitely many distinct $r_s$-representations.
\end{remark}
Indeed, taking into account Lemma~\ref{periodic_infinite}, we have
\[\Delta^{r_s}_{\alpha_1\ldots\alpha_n(s-1)}=
\Delta^{r_s}_{\alpha_1\ldots\alpha_ns(0)}=
\Delta^{r_s}_{\alpha_1\ldots\alpha_n[s-1]\ldots[s-1]s(0)},\]
\[\Delta^{r_s}_{\alpha_1\ldots\alpha_{n(r-s+1)}}=
\Delta^{r_s}_{\alpha_1\ldots\alpha_n[r-s](r)}=
\Delta^{r_s}_{\alpha_1\ldots\alpha_n[r-s+1]\ldots[r-s+1][r-s](r)}.\]

\begin{lemma}\label{form_zeros_at_start}
Let $x_0 \in \left(0,\frac{r}{s-1}\right)$ be a number with a unique $r_s$-representation
$\Delta^{r_s}_{\alpha_1 \alpha_2 \ldots \alpha_n \ldots}$.
If $\alpha_k = 0$ for some $k$, then $\alpha_i = 0$ for all $i < k$.
The representation of $x_0$ has the form
\[x_0=
\Delta^{r_s}_{0\ldots0\alpha_m\alpha_{m+1}\ldots\alpha_{m+n}\ldots},\] where  $0<\alpha_m<s,\;\; \alpha_{m+j}\in\{r-s+1, r-s+2, \ldots, s-1\}\equiv N_{r-s+1}^{s-1},$ 
and the sequence $(\alpha_{m+n})$ cannot have period $(s-1)$ or $(r-s+1)$.
\end{lemma}
\begin{proof}
Since $x_0 \neq 0$, its $r_s$-representation contains at least one nonzero digit.
The number $x_0$ has a unique $r_s$-representation; therefore, among all pairs of
adjacent digits in its representation there is no pair $(a,0)$ admitting an
alternative replacement $(a-1,s)$. Hence, if $\alpha_k = 0$ and
$\alpha_{k+1} \neq 0$, then $\alpha_{k+j} \neq 0$ for every $j \in \mathbb{N}$.
The digit $\alpha_{k+1}$ cannot exceed $s-1$, since otherwise the pair
$(0,\alpha_{k+1})$ would admit an alternative replacement
$(1,\alpha_{k+1}-s)$. Thus, $0 < \alpha_{k+1} \leq s-1$.

The uniqueness of the number $x_0$ representation guarantees the absence
of mutually replaceable pairs among consecutive digits of its $r_s$-representation, which is possible only if $r-s+1 \leq \alpha_{k+j} \leq s-1$.
Taking into account the above observations, the statement is proven.
\end{proof}

\begin{lemma}\label{form_r_at_start}
Let $x_0 \in \left(0,\frac{r}{s-1}\right)$ be a number with a unique $r_s$-representation
$\Delta^{r_s}_{\alpha_1 \alpha_2 \ldots \alpha_n \ldots}$.
If $\alpha_k = r$ for some $k$, then $\alpha_i = r$ for all $i < k$.
The representation of $x_0$ has the form
\[x_0=\Delta^{r_s}_{r\ldots r\alpha_m\alpha_{m+1}\ldots\alpha_{m+n}\ldots}, \mbox{ where } r-s<\alpha_m<r, \alpha_{m+i}\in N_{r-s+1}^{s-1},\] 
and the sequence $(\alpha_{m+n})$ cannot have period $(s-1)$ or $(r-s+1)$. 
\end{lemma}
\begin{proof}
Since $x_0 \neq \dfrac{r}{s-1}$, its $r_s$-representation contains digits different from $r$.
Assume that $\alpha_k = r$.
Suppose that there exists an index $i < k$ such that $\alpha_i < r$, and let
\[j=\max\{i:~\alpha_i<r,\; i<k\}.\]

Then the pair $(\alpha_j, \alpha_{j+1}) = (\alpha_j, r)$ admits an alternative replacement,
namely the pair $(\alpha_j + 1, r - s)$, which contradicts the uniqueness of the
$r_s$-representation of the number $x_0$.
Therefore, the $r_s$-representation of the number $x_0$ has the form
\[x_0 = \Delta^{r_s}_{r \ldots r \alpha_m \alpha_{m+1} \ldots}.\]

If we assume that $\alpha_m\leq r-s$, then the pair $(r,\alpha_m)$ admits an alternative replacement, namely the pair $(r-1,\alpha_m+s)$, which contradicts the uniqueness of the $r_s$-representation of the number $x_0$. Hence, $r-s<\alpha_m<r$.

Suppose that $\alpha_{m+j}<r-s+1$ або $\alpha_{m+j}>s-1$. Then, for one of the digit pairs $(\alpha_{m+j},\alpha_{m+j+1})$, an alternative replacement can be specified, which again leads to a contradiction to the uniqueness of the $r_s$-representation of $x_0$.

The same conclusion can be obtained on the basis of Lemma~\ref{uniq_inversed}, since the number
$x = \Delta^{r_s}_{r\ldots r \alpha_m \alpha_{m+1} \ldots}$
has a unique representation if and only if the number

    \[x_0'=\frac{r}{s-1}-x_0=
    \Delta^{r_s}_{0 \ldots 0[r-\alpha_m][r-\alpha_{m+1}][r-\alpha_{m+2}]\ldots}\]
has a unique $r_s$-representation.
According to Lemma~\ref{form_zeros_at_start}, the representation of $x_0'$ must satisfy the conditions
\[0 < r-\alpha_m < s \quad 
\text{and} 
\quad r-s+1 \leq r-\alpha_{m+1} \leq s-1,\]
which are equivalent to the conditions of the present statement.  
Taking into account the above observation, the lemma is proven.
\end{proof}

\begin{theorem}
The set $E$ of numbers $x \in \left(0,\frac{r}{s-1}\right)$ that have a unique $r_s$-representation consists precisely of those numbers whose
$r_s$-representations do not contain a period equal to $(s-1)$ or $(r-s+1)$ and have one of the following forms:
\[\Delta^{r_s}_{c \beta_1 \beta_2 \ldots \beta_n \ldots},\; \Delta^{r_s}_{00\ldots 0 a \beta_1 \beta_2 \ldots \beta_n \ldots},\; \Delta^{r_s}_{rr\ldots r b \beta_1 \beta_2 \ldots \beta_n \ldots},\]
where $ 0\leq a<s,\; r-s<b\leq r, \;\beta_i \in  N_{r-s+1}^{s-1}.$
\end{theorem}  
\begin{proof} 

If the $r_s$-representation of a number with a unique representation contains the digit $0$ or $r$, then, as established by the two preceding lemmas, the representation of the number has the form
\[\Delta^{r_s}_{00\ldots 0 a \beta_1\beta_2 \ldots \beta_n \ldots}
\quad \text{or} \quad
\Delta^{r_s}_{rr\ldots r b \beta_1\beta_2 \ldots \beta_n \ldots},\]
where $ 0\leq a<s$, $r-s<b\leq r$, $\beta_i \in  N_{r-s+1}^{s-1}$.

Consider the numbers that have a unique $r_s$-representation and do not use the digits $0$ and $r$ in their $r_s$-representations at all.
Let $x_0=\Delta^{r_s}_{\alpha_1\alpha_2\ldots\alpha_n\ldots}$ be a number of this type.
Suppose that for some $k>1$ the digit $\alpha_k\in\{1,\ldots,r-s\}\cup\{s,s+1,\ldots,r-1\}$, then the pair $(\alpha_{k-1},\alpha_k)$ admits an alternative replacement.
In the case $1\leq \alpha_k\leq r-s$, this replacement is the pair
$(\alpha_{k-1}-1,\alpha_{k}+s),$
and in the case $s\leq \alpha_{k}\leq r-1$ the alternative replacement is the pair $(\alpha_{k-1}+1,\alpha_k-s)$.
In both cases, this contradicts the uniqueness of the $r_s$-representation of the number $x_0$.
Hence,
$$\alpha_k \in N_{r-s+1}^{s-1}
\quad \text{for all } k\in\{2,3,\ldots\}.$$
The same conclusion can also be obtained by applying the criterion of uniqueness (Theorem~\ref{criterion_of_uniq}).
\end{proof}

\begin{theorem}
The set $E$ of numbers that have a unique $r_s$-representation is symmetric about the point $\frac{r}{2(s-1)}$, is not closed, and has Hausdorff--Besicovitch fractal dimension $\frac{\ln(2s-r-1)}{\ln s}$.
\end{theorem}
\begin{proof}
It is easy to verify that if a number 
$x=\Delta^{r_s}_{\alpha_1\ldots\alpha_n\ldots}$ 
belongs to the set $E$, then the number
$x'=\Delta^{r_s}_{[r-\alpha_1][r-\alpha_2]\ldots[r-\alpha_n]\ldots}$
also belongs to $E$, and conversely. Hence, the symmetry of the set $E$ is obvious.

Each point of the form
$\Delta^{r_s}_{\underbrace{0\ldots0}_k a\beta_1\beta_2\ldots\beta_n(c)}$,
where $c\in \{s-1,r-s+1\}$, $k \in \mathbb{N}$ is a limit point of the set $E$, but does not belong to it. Therefore, the set $E$ is not closed.

Let $D$ denote the set of all numbers whose $r_s$-representation has period $(s-1)$ or $(r-s+1)$. As noted above, $E \cap D = \varnothing$. It is well known that every countable set has Lebesgue measure zero and Hausdorff--Besicovitch dimension zero. Therefore, such sets can be neglected in metric problems.

Let us introduce the following notation. For $n \in \mathbb{N}\cup\{0\}$, define  
\[E_{0n}^a=\{x ~\colon~ 
x = \Delta^{r_s}_{\underbrace{ 0 \ldots 0}_n
a \beta_1 \beta_2 \ldots \beta_n \ldots},
\beta_n \in N_{r-s+1}^{s-1}\}, 0 < a < s;\]

\[E_{rn}^b=\{x ~\colon~ 
x = \Delta^{r_s}_{\underbrace{r \ldots r}_n
b \beta_1 \beta_2 \ldots \beta_n \ldots},
\beta_n \in N_{r-s+1}^{s-1}\}, r-s < b < r;\]

\[\overline{E}_{0n}^a \equiv E_{0n}^a \setminus D,\quad
  \overline{E}_{rn}^b \equiv E_{rn}^b \setminus D.\]

Then the set $E$ has the following structure:
\[E=\bigcup\limits_{n=0}^{\infty}
\left[
\bigcup\limits_{a=1}^{s-1} \overline{E}_{0n}^{a} 
\;\cup\;
\bigcup\limits_{b=r-s+1}^{r-1} \overline{E}_{rn}^{b}
\right].\]

All sets $E_{in}^{j}$ with fixed $n$ are isometric to each other and are self-similar. Each of them is similar to the set
\[C=\{x~\colon~ x=
\Delta^{r_s}_{\beta_1 \beta_2 \ldots \beta_n \ldots}, 
\beta_n \in N_{r-s+1}^{s-1}\},\]
which consists of numbers representable in the classical base-$s$ numeral system using digits from the set $N_{r-s+1}^{s-1}$.
This set is a self-similar Cantor-type set meeting the open set condition~\cite{6_Hutchinson}.
Therefore, its similarity dimension $\log_s(2s-r-1)$ coincides with its Hausdorff--Besicovitch dimension $\alpha_0(C)$.

The Hausdorff--Besicovitch dimensions of the sets $E$ and $C$ coincide due to the countable stability property of the dimension
 $(\alpha_0(\bigcup\limits_{i}F_i)=\sup\limits_{i}\alpha_0(F_i))$.
\end{proof}

\section{The set of numbers having a continuum of representations}
\begin{theorem}
A number $x_0$ for which some $r_s$-representation
$\Delta^{r_s}_{\alpha_1 \alpha_2 \ldots \alpha_n \ldots}$
has no period $(0)$ or $(r)$, but contains infinitely many digits from the set
\[
A_r \setminus N_{r-s+1}^{s-1}
= \{0,1,2,\ldots,r-s,\, s, s+1, \ldots, r-1, r\},
\]
admits a continuum of distinct $r_s$-representations.
\end{theorem}

\begin{proof}
The conditions of the lemma allow for the following cases:

\begin{enumerate}
\item The number contains infinitely many zeros;
\item The number contains infinitely many digits equal to $r$;
\item The number contains only finitely many digits equal to $0$ or $r$.
\end{enumerate}
We consider these cases separately.

\medskip
\noindent
\textbf{Case 1.}
Since the given $r_s$-representation of the number $x_0$ does not have the period $(0)$, the sequence $(\alpha_k)$ contains infinitely many pairs $(a,0)$ with $a \neq 0$ occurring as consecutive digits in the $r_s$-representation of $x_0$. However, each such pair admits an alternative replacement $(a-1,s)$. The existence of infinitely many such alternative replacements implies that the set of distinct $r_s$-representations of the number $x_0$ has the cardinality of the continuum.

\medskip
\noindent
\textbf{Case 2.}
An analogous situation occurs in the second case. Since the $r_s$-representation of the number $x_0$ does not have the period $(r)$, there exist infinitely many pairs $(a,r)$ with $a \neq r$ occurring as consecutive digits in the $r_s$-representation of $x_0$. Since this pair admits an alternative replacement, namely the pair $(a+1, r-s)$, we again obtain a continuum of distinct $r_s$-representations of the number $x_0$.

\medskip
\noindent
\textbf{Case 3.}
Suppose that the given $r_s$-representation of the number $x_0$ contains only finitely many digits equal to $0$ and $r$. 
Then there exists a digit $c$ belonging to one of the sets
$\{1,2,\ldots,r-s\}$ or $\{s,s+1,\ldots,r-1,r\}$ that occurs infinitely many times in the $r_s$-representation.

If the $r_s$-representation of $x_0$ is periodic with period $(c)$, where $c\neq 0$ and $c\neq r$, then, by Lemma~\ref{periodic_continuum}, the number $x_0$ has a continuum of distinct $r_s$-representations.

On the other hand, if the $r_s$-representation of $x_0$ is not periodic, then the sequence $(\alpha_k)$ contains infinitely many pairs $(b,c)$ of consecutive digits such that $c\neq b\neq 0, b\neq r$. If $c$ belongs to the set $\{1,2,\ldots,r-s\}$, then the pair $(b,c)$ admits the alternative replacement $(b-1,c+s)$. 
If $c$ belongs to $\{s,s+1,\ldots,r-1\}$, then the alternative replacement of $(b,c)$ is $(b+1,c-s)$. 
In both cases, the existence of infinitely many alternative replacements implies that the set of distinct $r_s$-representations of the number $x_0$ is a continuum.
\end{proof}

\begin{theorem} 
Almost all numbers with respect to Lebesgue measure on the interval
$\left(0,\frac{r}{s-1}\right)$
have a continuum of distinct $r_s$-representations; that is, the set of such
numbers has Lebesgue measure equal to $\frac{r}{s-1}$.
\end{theorem}

\begin{proof}
Let $B$ denote the set of all numbers for which at least one $r_s$-representation does not satisfy the conditions of the preceding theorem; that is, among the digits of such an $r_s$-representation, only finitely many belong to the set $A_r\setminus N_{r-s+1}^{s-1}$, while all the remaining digits belong to the set $\{r-s+1,r-s+2,\ldots,s-1\}$.

Denote by $B_n$ the set of all numbers whose $r_s$-representations contain no digits from the set
$A_r \setminus N_{r-s+1}^{s-1}$ starting from the $n$-th position, where $n \in \mathbb{N}$.

Each of the sets $B_n$ has Lebesgue measure zero, since for any fixed $n\in\mathbb{N}$ and any 
\mbox{$(\alpha_1,\alpha_2,\ldots,\alpha_n)\in A_r^n$} the set

\[\Delta^{r_s}_{\alpha_1 \ldots \alpha_n}\cap B_n,\quad 
\Delta^{r_s}_{\alpha_1 \ldots \alpha_n} =
\left[\sum\limits_{i=1}^{n} \frac{\alpha_i}{s^i};
\frac{r}{s^n(s-1)}+\sum\limits_{i=1}^{n}\frac{\alpha_i}{s^i}\right]\]
is nowhere dense and has Lebesgue measure zero.

Оbviously, 
\[B_1\subset B_2\subset\ldots\subset B_n\subset B_{n+1}\subset\ldots\]
and
\[B=\bigcup\limits_{n=1}^{\infty}B_n=\lim\limits_{n\to\infty}B_n.\]
Hence, for the Lebesgue measure we have $\lambda(B)=\lambda(\bigcup\limits_{n=1}^{\infty}B_n)\leq
\sum\limits_{n=1}^{\infty}\lambda(B_n)=0$. Therefore, the set of numbers whose $r_s$-representations satisfy the conditions
of the theorem has full Lebesgue measure.
\end{proof}
\begin{corollary}
The set of numbers with finitely or countably many distinct
$r_s$-representations is of Lebesgue measure zero.
\end{corollary}
\begin{remark}
The property for a number $x_0 \in \left[0,\frac{r}{s-1}\right]$ to have a continuum of distinct representations is typical.
\end{remark}

% include bibliography:
\bibliographystyle{plainurl}
\bibliography{mybiblio-ed2}

\end{document}